\documentclass{article}
\usepackage[latin1]{inputenc}
\usepackage{amsmath}
\usepackage{amsthm}
\usepackage{amsfonts}
\usepackage{amssymb}
\usepackage{amstext}
\usepackage{amsgen}
\usepackage{amsbsy}
\usepackage{amsopn}
\newtheorem{theorem}{Theorem}[section]
\newtheorem{lemma}[theorem]{Lemma}
\newtheorem{proposition}[theorem]{Proposition}

\newtheorem{definition}[theorem]{Definition}

\newtheorem{corollary}[theorem]{Corollary}

\title{On positivity in algebras of tempered generalized functions}
\author{Eberhard Mayerhofer \footnote{ Vienna Institute of Finance,
Heiligenst\"adterstrasse 46-48, 1190 Vienna, Austria. Supported by
grants P16742-N04 and Y237-N13 hold by M. Kunzinger, University of
Vienna}}
\date{Winter 2006}
\begin{document}
\maketitle
\begin{abstract}
An explicit counterexample shows that contrary to the situation in the special Colombeau algebra,
positivity and invertibility cannot be characterized pointwise
in algebras of tempered generalized functions. Further a point value characterization of the latter is refined.
\end{abstract}
\section{Introduction}\label{chapterconstant}
Let $d\geq 1$ and suppose a non-empty open subset $\Omega$ of
$\mathbb R^d$ is given. We denote by $\mathcal G_\tau(\Omega)$ the
algebra of tempered generalized functions on $\Omega$. It has been
established in \cite{MO} that $\mathcal G_\tau(\Omega)$ admits a
point-value characterization whenever $\Omega$ is a box. It has
further been shown that there exist open sets $\Omega\subset\mathbb
R^d$ with infinitely connected components, such that elements of
$\mathcal G_\tau(\Omega)$ are not determined by evaluation at
moderate generalized points. Aim of this note is to discuss the
question whether in algebras $\mathcal G_\tau$ that admit a
point-value characterization, a point-value characterization of
invertibility can be given. For an introduction to generalized
function algebras as introduced by Colombeau, Rosinger, Egorov and
others, we refer to the standard reference \cite{bible}.
\section{Preliminaries}
Let $\Omega$ be a non-empty open subset of $\mathbb R^d$. We work in
the Colombeau algebra of tempered generalized functions on $\Omega$
given by the quotient
\[
\mathcal G_\tau(\Omega):=\mathcal E_{M,\tau}(\Omega)/\mathcal N_\tau(\Omega),
\]
where  the ring of tempered moderate nets of smooth functions is given by
\[
\mathcal
E_{M,\tau}(\Omega):=\{(u_\varepsilon)_\varepsilon\in\mathcal
C^\infty(\Omega)^{(0,1]}\,|\,\exists\, N:\sup_{x\in\Omega}\vert
u_\varepsilon(x)\vert=O(\varepsilon^{-N}(1+\vert x\vert)^N))\}
\]
whereas the ideal of tempered negligible functions is given by
\[
\mathcal N_\tau(\Omega):=\{(u_\varepsilon)_\varepsilon\in\mathcal
C^\infty(\Omega)^{(0,1]}\,|\,\exists\, N\,\forall\,
p:\sup_{x\in\Omega}\vert u_\varepsilon(x)\vert=O(\varepsilon^p
(1+\vert x\vert)^N))\}.
\]
The latter is an ideal in $\mathcal E_{M,\tau}(\Omega)$.

Moderate generalized points in $\Omega$ are given by
\[
\widetilde{\Omega}:=\Omega_M/\sim
\]
with
\[
\Omega_M:=
\{(x_\varepsilon)_\varepsilon\in\Omega^{(0,1]}\,|\,\exists\,N:\vert
x_\varepsilon\vert =O(\varepsilon^{-N})\}
\]
and $\sim$ is the equivalence relation on $\Omega_M$ defined by
\[
(x_\varepsilon)_\varepsilon\sim (y_\varepsilon)_\varepsilon\Leftrightarrow \forall\,p\geq 0:\; \vert x_\varepsilon-y_\varepsilon\vert=O(\varepsilon^p),\,(\varepsilon\rightarrow 0).
\]
Tempered generalized functions can be evaluated on
moderate generalized points, that is given $u\in\mathcal
G_\tau(\Omega)$ and $\widetilde x\in\widetilde{\Omega}$, with
representatives $(u_\varepsilon)_\varepsilon$ and
$(x_\varepsilon)_\varepsilon$ respectively,
$(u_\varepsilon(x_\varepsilon))+\mathcal N_\tau(\Omega)$ yields a
well defined element of $\widetilde{\mathbb R}$. The following is
established in \cite{MO}:
\begin{theorem}\label{charpointvalues}
Let $\Omega\subseteq \mathbb R^d$ be an open box. Let $u\in\mathcal
G_\tau(\Omega)$. The following are equivalent:
\begin{enumerate}
\item $u=0$ in $\mathcal G_\tau(\Omega)$,
\item \label{charpoints} For all $\widetilde x\in\widetilde{\Omega}$ we have $u(\widetilde x)=0$ in $\widetilde{\mathbb R}$.
\end{enumerate}
\end{theorem}
Finally, we shall need the notion of positivity and strictly non-negativity of generalized numbers. For a new characterization
of these properties we refer to \cite{algfound}.
\begin{definition}
$x\in\widetilde{\mathbb R}$ is called strictly non-zero (resp.\ strictly positive), if for each representative
$(x_\varepsilon)_\varepsilon$ of $x$ we have
\[
(\exists m\geq 0)(\exists \varepsilon_0\in (0,1])(\forall \varepsilon<\varepsilon_0, \vert x_\varepsilon\vert>\varepsilon^m (\mbox{resp}.\ x_\varepsilon>\varepsilon^m)).
\]
\end{definition}
\subsection*{Motivation of the paper}
It is well known that in the special algebra $\mathcal G^s(\Omega)$ based on an arbitrary open set $\Omega\subset\mathbb R^d$, generalized functions not only are uniquely determined by evaluation on so-called compactly supported generalized points, but also a point value characterization of invertibility as well as positivity
is available (\cite{algfound,MO}). Algebras of tempered generalized functions, $\mathcal G_\tau(\Omega)$, however, not always admit a point value characterization.
However, even if they do (see above Theorem \ref{charpointvalues}), invertibility has not yet been understood pointwise. This open problem is discussed in the following section and a negative answer is given for open boxes $\Omega$ in $\mathbb R^d$.
\section{Point-wise invertibility in $\mathcal G_\tau$}
To start with we state the basic lemma:
\begin{lemma}
Let $m\in\mathbb R$. The map $s_m:\mathcal G_\tau(\mathbb R^d)\rightarrow\mathcal G_\tau(\mathbb R^d)$
given by
\[
u\mapsto u_\varepsilon(\varepsilon^m x)+\mathcal N_\tau(\mathbb R^d)
\]
is well defined. Furthermore, $s_m$ is an algebra isomorphism for each $m$.
\end{lemma}

\begin{theorem}
Let $u\in\mathcal G_\tau(\Omega)$. The following are equivalent:
\begin{enumerate}
\item For all $\widetilde x\in\widetilde{\Omega}$, $u(\widetilde x)$ is invertible.
\item For each $m\geq 0$, $s_m(u)$ is strictly non-zero on the unit ball, that is,
for each $m\geq 0$ there exists $N(m)\geq 0$ such that
\[
\inf_{\vert x\vert\leq 1}\vert u_\varepsilon(\varepsilon^m x)\vert\geq \varepsilon^{N(m)}
\]
for sufficiently small $\varepsilon$.
\end{enumerate}
\end{theorem}
Hence we have translated point-wise invertibility of $u$ into a countable number of local conditions of
tempered generalized functions derived from $u$, namely of $s_m(u)$,  $m\in\mathbb N_0$. The following section is dedicated to showing
that these local conditions indeed do not suffice to guarantee invertibility of $u$ in $\mathcal G_\tau(\mathbb R^d)$.
\section{Global invertibility cannot be characterized point-wise}
We begin by characterizing invertibility of generalized functions by
evaluation at generalized points for bounded domains $\Omega$.
\begin{lemma}\label{lemmabounded}
Let $\Omega\subset \mathbb R^d$ be a bounded box, $u\in\mathcal G_\tau(\Omega)$. The following are equivalent:
\begin{enumerate}
\item \label{propertyone} $u$ is invertible in  $u\in\mathcal G_\tau(\Omega)$,
\item \label{propertytwo} For all $\widetilde x\in\widetilde{\Omega}$, $u(\widetilde x)$ is invertible.
\end{enumerate}
\end{lemma}
\begin{proof}
(\ref{propertyone})$\Rightarrow$(\ref{propertytwo}) is evident. To
show (\ref{propertytwo})$\Rightarrow$(\ref{propertyone}), we observe
first that $\widetilde{\Omega}=\widetilde{\Omega}$, since $\Omega$
is bounded. Assume, by contradiction that $u$ is not invertible. Let
$(u_\varepsilon)_\varepsilon$ be a representative of $u$. Since $u$
is not invertible, $u_\varepsilon$ cannot be bounded from below by a
fixed power of $\varepsilon$. Hence, there exists
$(x_k)_k\in\Omega^{\mathbb N_0}$ and $\varepsilon_k\rightarrow 0$
such that $\vert u_{\varepsilon_k}(x_k)\vert<\varepsilon_k^k$.
Define a generalized point $\widetilde x$ via its net
$(x_\varepsilon)_\varepsilon$ by $x_\varepsilon:=x_k$ whenever
$\varepsilon\in(\varepsilon_{k},\varepsilon_{k-1}]$, $k\geq 1$. Then
$u(\widetilde x)$ is not invertible in $\widetilde{\mathbb R}$,
since it is not strictly non-zero. This contradicts
(\ref{propertytwo}) and we are done.
\end{proof}
The main aim of this section is to show:
\begin{theorem}\label{counterexample}
$\mathcal G_\tau(\mathbb R^d)$ does not admit a point value
characterization of invertibility: there exist non-invertible
functions $u\in\mathcal G_\tau(\mathbb R^d)$ such that $u(\widetilde
x)$ is invertible for all $\widetilde x\in\widetilde{\Omega}$.
\end{theorem}
Before we provide a proof of this statement, we investigate the underlying counterexample:
\begin{proposition}\label{counterexample1}
Let $\sigma\in C^{\infty}(\mathbb R^d)$ be a function such that
$1-\sigma$ is a cutoff at $x=0$,  $\sigma=1$ on $\vert x\vert\leq
1/2$ and $\sigma=0$ on $\vert x\vert\geq 1$. For each
$(x,\varepsilon)\in\mathbb R^d\times (0,1]$ define
\[
u_\varepsilon(x):=(1-\sigma(x))g_\varepsilon(x)+\sigma(x),
\]
with $g_\varepsilon(x):=\varepsilon^{(\log_\varepsilon(\vert x\vert))^2}$. Then we have the following:
\begin{enumerate}
\item \label{prop1} $(u_\varepsilon)_\varepsilon$ is moderate, that is, $(u_\varepsilon)_\varepsilon\in\mathcal E_{M,\tau}(\mathbb R^d)$.
\item \label{prop2} $(u_\varepsilon)_\varepsilon\notin\mathcal N_\tau(\mathbb R^d)$.
\item \label{prop3} $(u_\varepsilon)_\varepsilon$ is strictly non-zero  for all $x$ in $\vert x\vert<\varepsilon^{-j}$ for $j=0,1,2,\dots$, more precisely,
the following three estimates hold
\begin{equation}\label{61}
(\forall  \varepsilon\neq 1)(\forall j\geq 0)(\forall
x:\varepsilon^{-j}\leq \vert
x\vert<\varepsilon^{-j-1})(\varepsilon^{(j+1)^2}<\vert
u_\varepsilon(x)\vert=u_\varepsilon(x)\leq \varepsilon^{j^2}),
\end{equation}
\begin{equation}\label{62}
(\forall \varepsilon>0)(\forall x,\,\vert x\vert\leq
1/2)(u_\varepsilon(x)=1),
\end{equation}
\begin{equation}\label{63}
(\forall \varepsilon<1/2)(\forall x,1/2\leq\vert x\vert\leq 1)(1\leq
g_\varepsilon(x)<2,\; 1\leq u_\varepsilon(x)<3).
\end{equation}
\end{enumerate}
\end{proposition}
\begin{proof}
First, it is clear that $g_\varepsilon$ is smooth away from zero,
hence $u_\varepsilon\in C^\infty(\mathbb R^d)$ for each
$\varepsilon\in(0,1]$. We start by proving (\ref{prop2}): For $j\geq
0$ we define $x_\varepsilon^{(j)}:=\varepsilon^{-j}$. We have
$g_\varepsilon(x_\varepsilon^{(j)})=\varepsilon^{j^2}$, hence
$u_\varepsilon\notin\mathcal N_\tau(\mathbb R^d)$. Proof of
(\ref{prop3}): First, we show (\ref{61}). Let $j\geq 0$, then we
have whenever $0<\varepsilon<1$
\begin{eqnarray}\\\nonumber
\varepsilon^{-j}&\leq&\vert x\vert<\varepsilon^{-j-1}\Rightarrow-j\log\varepsilon\leq\log(\vert x\vert)<-(j+1)\log\varepsilon\\\nonumber
\Rightarrow j&\leq& \frac{\log (\vert x\vert)}{-\log\varepsilon}=-\log_\varepsilon(\vert x\vert)<j+1\Rightarrow j^2\leq (\log_\varepsilon(\vert x\vert))^2<(j+1)^2\\\nonumber
\Rightarrow \varepsilon^{(j+1)^2}&<&g_\varepsilon(x)=u_\varepsilon(x)\leq \varepsilon^{j^2}.
\end{eqnarray}
And the last equality holds because $\vert x\vert >1$ for
$\varepsilon\neq 1$. Estimate (\ref{62}) follows directly from the
choice of the cutoff function $1-\sigma$. To see (\ref{63}), we
check that following implications hold for $\varepsilon<1/2$:
\begin{eqnarray}\\\nonumber
1/2&\leq&\vert x\vert\leq 1\Rightarrow-\log 2<\log(\vert x\vert)\leq 0\Rightarrow \frac{\log 2}{\log\varepsilon}<-\frac{\log(\vert x\vert)}{\log\varepsilon}\leq 0\\\nonumber
\Rightarrow -\left(\frac{\log 2}{\log \varepsilon}\right)^{\frac{1}{2}}&<&-\frac{\log(\vert x\vert)}{\log\varepsilon}\leq 0 \Rightarrow 1\leq g_\varepsilon(x)<2
\Rightarrow 1\leq u_\varepsilon(x)<3.
\end{eqnarray}
Thus we have finished the proof of (\ref{prop3}). Finally, we are prepared to show (\ref{prop1}). To prove that $(u_\varepsilon)_\varepsilon$ is moderate, it suffices to establish moderate estimates of the latter for $\vert x\vert\geq 1/2$ only, because $u_\varepsilon(x)=1$ on $\vert x\vert\leq 1/2$ according to (\ref{62}).
\begin{itemize}
\item  We start with the zero-order erstimates. By (\ref{61})--(\ref{63}), we have  for all $\varepsilon<1/2$ and for all $x\in\mathbb R^d$, $0<u_\varepsilon(x)<3$.
\item  Derivatives of first order: Let $i=1,\dots,d$. For $\vert x\vert\geq 1/2$ one has
\[
\partial_i g_\varepsilon(x)=2 g_\varepsilon(x)\frac{x_i\log(\vert x\vert)}{\vert x\vert ^2\log\varepsilon}
\]
Since $\vert g_\varepsilon(x)\vert <2$, and $\vert x\vert \geq 1/2$
we therefore have for sufficiently small $\varepsilon$,
\[
\vert\partial_i g_\varepsilon(x)\vert\leq 16 x_i\frac{\log (\vert x\vert)}{\log\varepsilon}<16 (1+\vert x\vert)^2\varepsilon^{-2},
\]
so we have derived moderate bounds for the first derivative.
\item Estimates for higher order derivatives can be obtained similarly.
\end{itemize}
\end{proof}
\begin{proof}[Proof of Theorem \ref{counterexample}]
We use $(u_\varepsilon)_\varepsilon$ as defined in the preceding statement, Proposition \ref{counterexample1}. According to the latter, $u:=[(u_\varepsilon)_\varepsilon]$ is a well defined element of $\mathcal G_\tau(\mathbb R^d)$. Assume now $u$ is invertible, that is there exists $(v_\varepsilon)_\varepsilon\in\mathcal E_{M,\tau}(\mathbb R^d)$ such that for some $N\geq 0$,
\begin{equation}\label{estimate1}
\vert v_\varepsilon(x)\vert\leq \varepsilon^{-N}(1+\vert x\vert)^N
\end{equation}
for sufficiently small $\varepsilon$ and that
\begin{equation}\label{lasteq}
u_\varepsilon(x) v_\varepsilon(x)=1+n_\varepsilon(x)
\end{equation}
for some $(n_\varepsilon)_\varepsilon\in\mathcal N_\tau(\mathbb R^d)$. By construction of $u_\varepsilon$, we have for each $j\geq 0$,
$u_\varepsilon(\varepsilon^{-j})=\varepsilon^{j^2}$. Hence by (\ref{lasteq}) and (\ref{estimate1}) we have
\begin{equation}
\vert u_\varepsilon(\varepsilon^{-j})v_\varepsilon(\varepsilon^{-j})\vert\leq \varepsilon^{j^2-N}(1+\varepsilon^{-j})^N\leq \varepsilon^{j^2-N(1+2j)}.
\end{equation}
Since $j$ is arbitrary, we may set $j=3N+1$, then
\begin{equation}
\vert u_\varepsilon(\varepsilon^{-j})v_\varepsilon(\varepsilon^{-j})\vert\leq\varepsilon^1.
\end{equation}
However, this contradicts (\ref{lasteq}), because by negligibility of $(n_\varepsilon)_\varepsilon$ we have\newline $\vert n_\varepsilon(\varepsilon^{-(3N+1)})\vert=O(\varepsilon^p)$ for all $p\geq 0$.
Therefore $u$ is not invertible and we are done.
\end{proof}

As a consequence of Lemma \ref{lemmabounded} and Theorem \ref{counterexample} we have the following:
\begin{corollary}
Let $\Omega\subseteq\mathbb R^d$ be a box. The following are equivalent:
\begin{enumerate}
\item \label{invcharx} Invertiblity can be characerized pointwise in $\mathcal G_\tau(\Omega)$.
\item \label{invchary} $\Omega$ is a bounded.
\end{enumerate}
\end{corollary}
\begin{proof}
Since the proof of (\ref{invchary})$\Rightarrow$(\ref{invcharx}) is
provided by Lemma \ref{lemmabounded}, we only need to show the
converse direction. If $\Omega=\mathbb R^d$, then this is a
consequence of Theorem \ref{counterexample}. Assume therefore
$\Omega\neq \mathbb R^d$. By a permutation of variables and a
translation $x_1\mapsto x_1+t$ for some $t$ or a possible reflexion
$x_1\mapsto-x_1$, we may assume without loss of generality that
$\Omega=(0,\infty)\times \Omega'$ with $\Omega\subseteq\mathbb
R^{d-1}$. Let $1-\sigma\in\mathcal D(\mathbb R^1)$ such that
$\sigma=0$ on $ x\geq 1$ and $\sigma=1$  for $x\leq 1/2$. Define
similarly to Theorem \ref{counterexample}, for $x>0$,
$u^{(1)}_\varepsilon(x)=(1-\sigma(x))\varepsilon^{(\log_\varepsilon(x))^2}+\sigma(x)$
and let $u_\varepsilon(x_1,\dots,x_n):=u_\varepsilon^{(1)}(x_1)$.
Hence, when $\vert x\vert<\varepsilon^{-j}$, then evidently $
x_1<\varepsilon^{-j}$, hence by Theorem \ref{counterexample},
$u_\varepsilon(x)>\varepsilon^{j^2}$. Hence, for all $\widetilde
x\in\widetilde{\Omega}$, $u(\widetilde x)$ is strictly positive,
hence invertible. One can further show that $u$ cannot be invertible
by following the lines of the proof of Theorem \ref{counterexample}.
\end{proof}
\section{Point-values in $\mathcal G_\tau(\Omega)$}
Aim of this section is to answer the following question raised by
Stevan Pilipovi{\'c}: "Can
elements of $\mathcal G_\tau(\Omega)$ uniquely be determined merely
by evaluation at generalized points with strict positive distance to
the boundary $\partial \Omega$?".

This question only makes sense in algebras which admit a
point-value characterization; this for instance is the case for open
boxes in $\mathbb R^d$ (cf.\ Theorem \ref{charpointvalues}). Hence,
we shall discuss this case here.

First we need some technical lemma in order to allow for a well
defined notion of "distance of a generalized point $\widetilde x$ to
the boundary $\partial \Omega$".
\begin{lemma}
Let $a_1,\dots,a_m\in\widetilde{\mathbb R}$ be given ($m\geq 1$).
Then
\[
\inf
(a_1,\dots,a_m):=(\min(a_1^\varepsilon,\dots,a_m^\varepsilon))_\varepsilon+\mathcal
N,
\]
where $(a_i^\varepsilon)_\varepsilon$ ($i=1,\dots,m$) are
representatives of $a_i$ ($i=1,\dots,m$), yields a well-defined
element of $\widetilde{\mathbb R}$.
\end{lemma}
\begin{proof}
Assume this is not the case, that is, there exist representatives
$(a_1^\varepsilon)_\varepsilon,\,\dots,(a_m^\varepsilon)_\varepsilon$
of $a_1,\dots,a_m$ and negligible nets of numbers
$(n_1^\varepsilon)_\varepsilon,\,\dots,(n_m^\varepsilon)_\varepsilon$,
a zero sequence $\varepsilon_k\rightarrow 0$ and an $m\geq 0$ such
that
\begin{equation}\label{eqpig}
(\forall k) (\vert \min (a_i^{\varepsilon_k};i=1,\dots,m)-\min
(a_i^{\varepsilon_k}+n_i^{\varepsilon_k};i=1,\dots,m)\vert\geq\varepsilon_k^m
.
\end{equation}
Hence, by the pigeon hole principle, there exists an infinite
subsequence $k_l$, $l=1,2,\dots,$ as well as
$(i,j)\in\{1,\dots,m\}^2$ such that
\begin{equation}\label{piggie}
a_i^{\varepsilon_{k_l}}=\min
(a_i^{\varepsilon_{k_l}};i=1,\dots,m)
\end{equation}
as well as
\begin{equation}\label{piggy}
a_j^{\varepsilon_{k_l}}+n_j^{\varepsilon_{k_l}}=\min
(a_i^{\varepsilon_{k_l}}+n_i^{\varepsilon_{k_l}};i=1,\dots,m).
\end{equation}
We distinguish the following two cases.
\begin{itemize}
\item If $i=j$, then by (\ref{eqpig}), $\vert a_i^{\varepsilon_{k_l}}-(a_i^{\varepsilon_{k_l}}+n_i^{\varepsilon_{k_l}}) \vert=\vert n_i^{\varepsilon_{k_l}}\vert=O(\varepsilon_{k_l}^p)$ ($l\rightarrow\infty$) for each $p$, hence this contradicts (\ref{eqpig}). Hence, the
sequence $k_l$ cannot be infinite, contradiction.
\item If $i\neq j$, then by (\ref{eqpig}) for large $l$ we even
have, by the negligibility of $n_i^\varepsilon$,
\begin{equation}\label{pigex}
\vert a_i^{\varepsilon_{k_l}}-a_j^{\varepsilon_{k_l}}
\vert>\varepsilon_{k_l}^m/2.
\end{equation}
Because of (\ref{piggie}) we have
\[
a_i^{\varepsilon_{k_l}}\leq a_j^{\varepsilon_{k_l}}.
\]
Using further (\ref{pigex}), the latter inequality yields
\begin{equation}\label{salat1}
a_i^{\varepsilon_{k_l}}<
a_j^{\varepsilon_{k_l}}-\varepsilon_{k_l}^m/2.
\end{equation}
Furthermore, by (\ref{piggy}) we got
\begin{equation}\label{salat2}
a_j^{\varepsilon_{k_l}}\leq
a_i^{\varepsilon_{k_l}}+n_i^{\varepsilon_{k_l}}-n_j^{\varepsilon_{k_l}}.
\end{equation}
Hence by using (\ref{salat1}) and (\ref{salat2}) and the
negligibility of $n_i^{\varepsilon_{k_l}}-n_j^{\varepsilon_{k_l}}$,
we receive for arbitrary $p\geq 0$ for sufficiently large $l$
\[
a_i^{\varepsilon_{k_l}}+\varepsilon_{k_l}^m/2<a_j^{\varepsilon_{k_l}}\leq
a_i^{\varepsilon_{k_l}}+n_i^{\varepsilon_{k_l}}-n_j^{\varepsilon_{k_l}}<\varepsilon_{k_l}^p.
\]
This is a contradiction for $p>m$. Hence $k_l$ cannot be an infinite
sequence. This contradiction proves the claim.

\end{itemize}

\end{proof}
\begin{definition}\label{defdist}
Let $\Omega$ be a box in $\mathbb R^d$, $\Omega\neq\mathbb R^d$. For
$\widetilde x\in\widetilde{\Omega}$ we define the distance of
$\widetilde x$ to the boundary by
\[
d(\widetilde
x,\partial\Omega):=(d(x_\varepsilon,\partial\Omega))_\varepsilon+\mathcal
N,
\]
where $(x_\varepsilon)_\varepsilon$ is an arbitrary representative
of $x$. \end{definition}
\begin{lemma}
The distance function as given in Definition \ref{defdist} is
well-defined.
\end{lemma}
\begin{proof}
Let $\widetilde x\in\widetilde{\Omega}$ with representative $(x_\varepsilon)_\varepsilon$
be given. We can write $\Omega=\prod_{j=1}^d I_j$ with
$I_j:=(a_j,b_j)$, $a_j,b_j\in\mathbb R\cup\{-\infty,\infty\}$.
Clearly for each $\varepsilon$ we have
\[
d(x_\varepsilon,\partial\Omega)=\min\{\vert
x_\varepsilon^{(j)}-a_j\vert,\vert
x_\varepsilon^{(j)}-b_j\vert,\;a_j,\;b_j\neq \pm\infty\}.
\]
Hence, applying the preceding lemma to the the generalized numbers
$[(\vert x^{(i)}_\varepsilon-a_i\vert)_\varepsilon]$, $[(\vert
x^{(i)}_\varepsilon-b_i\vert)_\varepsilon]$ for
$a_i,b_i\neq\pm\infty$, we have the assertion.
\end{proof}
\begin{definition}
Let $\widetilde x,\;\widetilde y\in\widetilde{\mathbb R}$. We say
$\widetilde x$ is strictly smaller then $\widetilde y$, if for all
representatives
$(x_\varepsilon)_\varepsilon,\;(y_\varepsilon)_\varepsilon$ there
exists $m$ such that $y_\varepsilon-x_\varepsilon\geq \varepsilon^m$
for sufficiently small $\varepsilon$.
\end{definition}
Now we come to our main statement which answers S. Pilipovic's
question for boxes:
\begin{theorem}
Let $\emptyset\neq\Omega\subsetneq\mathbb R^d$ be a box,
$u\in\mathcal G_\tau(\Omega)$. The following are equivalent:
\begin{enumerate}
\item  \label{lab1} $u=0$ in $\mathcal G_\tau(\Omega)$,
\item \label{lab2} for all $\widetilde x\in\widetilde{\Omega}$ with $d(\widetilde
x,\partial\Omega)>0$ we have $u(\widetilde x)=0$ in
$\widetilde{\mathbb R}$.
\end{enumerate}
\end{theorem}
\begin{proof}
Since (\ref{lab1})$\Rightarrow$(\ref{lab2}) is clear, we only need
to show the converse direction. Let
$(u_\varepsilon)_\varepsilon,\;(x_\varepsilon)_\varepsilon$ be
representatives of $u$ and $\widetilde x$. Let $m>0$. Then we have
\begin{eqnarray}\label{eq1x}
&&(\exists \varepsilon_0)(\forall\varepsilon<\varepsilon)(\forall
i=1,\dots,d)(\exists\delta_\varepsilon^{(i)}\in
\{-1,1\})\\\nonumber&&(\textit{for}\;\;y_\varepsilon^{(i)}:=x_\varepsilon^{(i)}+\delta_\varepsilon^{(i)}\varepsilon^me_i\;\;\textit{we
have}\;\;\\\nonumber&&y_\varepsilon\in \Omega\;\;\\\nonumber
&&\textit{and}\\\nonumber&&
d(y_\varepsilon,\partial\Omega)>\varepsilon^m).
\end{eqnarray}
 In
terms of representatives, condition (\ref{lab2}) means
\begin{equation}
(\forall p\geq 0)(\vert
u_\varepsilon(y_\varepsilon)\vert=O(\varepsilon^p),\;
\varepsilon\rightarrow 0).
\end{equation}
Furthermore, by the mean value theorem,
\begin{equation}
(\forall\varepsilon<\varepsilon_0)(\exists
\Theta_\varepsilon\in(0,1))
(u_\varepsilon(y_\varepsilon)-u_\varepsilon(x_\varepsilon)=\sum_{i=^1}^d\{\partial_i
u_\varepsilon(x_\varepsilon+\Theta_\varepsilon(y_\varepsilon-x_\varepsilon))\}\delta_\varepsilon^{(i)}\varepsilon^m)
.
\end{equation}
Clearly, $(z_\varepsilon)_\varepsilon$ defined by
$z_\varepsilon:=x_\varepsilon+\Theta_\varepsilon(y_\varepsilon-x_\varepsilon)$
has moderate growth, say $\vert z_\varepsilon\vert\leq
\varepsilon^{-N'}$ for some $N'$ and sufficiently small
$\varepsilon$. Furthermore, since $u_\varepsilon$ is moderate, for
some $N$ and sufficiently small $\varepsilon$ we further have
\begin{equation}\label{eq4x}
\vert u_\varepsilon(z_\varepsilon)\vert\leq\varepsilon^{-N}(1+\vert
z_\varepsilon\vert)^N\leq \varepsilon^{-N(N'+1)}.
\end{equation}
Putting eq.\ (\ref{eq1x})--(\ref{eq4x}) together, we obtain for
arbitrary $p,\,m$ and sufficiently small $\varepsilon$,
\begin{eqnarray}
&&\vert \vert
u_\varepsilon(x_\varepsilon)\vert-\varepsilon^p\vert\leq\vert
u_\varepsilon(y_\varepsilon)-u_\varepsilon(x_\varepsilon)\vert\\\nonumber
&&\leq \varepsilon^{m-N(2+N')}.
\end{eqnarray}
Hence $\forall\, l\geq 0,\,\vert
u_\varepsilon(x_\varepsilon)\vert=O(\varepsilon^l)$,
($\varepsilon\rightarrow 0$). Hence we have proven that for all
$\widetilde x\in\widetilde{\Omega}$ we have $u(\widetilde x)=0$.
According to Theorem \ref{charpointvalues} (\ref{charpoints}), $u=0$
in $\mathcal G_\tau(\Omega)$ and we are done.
\end{proof}
\section*{Background Story and Outlook}
This manuscript solves elementary questions raised by Michael
Kunzinger (in the context of generalized Variational Calculus) and
Stevan Pilipovic. It was written in winter 2006 and presented in
Innsbruck some weeks later; since then, pointwise characterizations
in generalized function algebras have been advanced, e.g., by
Vernaeve \cite{Vernaeve} (cf.\ the proof of Proposition 3.4). The
author has also contributed to some further work by Pilipovic et al
\cite{scarpi}, which uses similar conditions as Proposition 3.5 (3)
in \cite{Vernaeve}. For related work by the author himself, cf.
\cite{algfound, egorov, spherical}.

\end{document}